\documentclass[a4paper,12pt]{amsart}
\usepackage{amssymb,amsfonts,amsmath,amsthm}
\usepackage{amsrefs,mathtools,stmaryrd}
\usepackage{xcolor}
\usepackage{tikz}
\usepackage{verbatim}
\usepackage[english]{babel}
\usepackage[all]{xy}
\usepackage{doi}
\usepackage{booktabs}
\usepackage{tabularx, makecell}
\newcolumntype{L}{>{\raggedright\arraybackslash}X}
\newcolumntype{R}{>{\raggedleft\arraybackslash}X}
\usepackage{etoolbox}
\usepackage{shuffle}

\usepackage[margin=0.9in]{geometry}

\def\={\; = \;}
\def\+{\; + \;}

\def\:{\; \colon \;}

\def\arXiv#1{arXiv:\href{http://arXiv.org/abs/#1}{#1}}

\newcommand{\ol}[1]{\overline{#1}}
\newcommand{\es}[1]{\; #1 \;}

\newcommand{\eps}{\varepsilon}
\newcommand{\sm}{\smallsetminus}

\newcommand{\re}{\operatorname{Re}}

\newcommand{\RR}{{\mathbb{R}}}
\newcommand{\DD}{{\mathbb{D}}}
\newcommand{\ZZ}{{\mathbb{Z}}}
\newcommand{\CC}{{\mathbb{C}}}

\newcommand{\QQ}{{\mathbb{Q}}}

\newcommand{\Pp}{{\mathcal{P}}}

\newcommand{\MZV}{{\mathfrak{Z}}}
\newcommand{\Z}{\zeta}

\newcommand{\ang}[1]{\langle #1\rangle}

\newcommand{\Area}{\operatorname{Area}}

\newcommand{\dl}{\partial}

\newcommand{\Li}{\mathrm{Li}}

\renewcommand{\phi}{\varphi}

\theoremstyle{plain}
\newtheorem{theorem}{Theorem}
\newtheorem{lemma}{Lemma}
\newtheorem{proposition}{Proposition}
\newtheorem{conjecture}{Conjecture}
\newtheorem{corollary}{Corollary}
\theoremstyle{definition}
\newtheorem{remark}{Remark}

\title{On Dirichlet eigenvalues of regular polygons}
\author[David~Berghaus]{David Berghaus}
\address{Bethe Center, Bonn University, Nussallee 12, 53115 Bonn, Germany}
\email{berghaus@th.physik.uni-bonn.de}
\author[Bogdan~Georgiev]{Bogdan~Georgiev}
\address{Fraunhofer IAIS, Institute for Intelligent Analysis and Information  Systems, 53757 Sankt Augustin, Germany}
\email{bogdan.m.georgiev@gmail.com}
\author[Hartmut~Monien]{Hartmut~Monien}
\address{Bethe Center, Bonn University, Nussallee 12, 53115 Bonn, Germany}
\email{hmonien@uni-bonn.de}
\author[Danylo~Radchenko]{Danylo~Radchenko}
\address{ETH Zurich, Mathematics Department, Zurich 8092, Switzerland}
\email{danradchenko@gmail.com}

\begin{document}
\maketitle
\begin{abstract}
	We prove that the first Dirichlet eigenvalue of a regular $N$-gon
	of area $\pi$ has an asymptotic expansion of the form $\lambda_1(1+\sum_{n\ge 3}\frac{C_n(\lambda_1)}{N^n})$ as $N\to\infty$, 
	where $\lambda_1$ is the first Dirichlet eigenvalue of the unit disk 
	and $C_n$ are polynomials whose coefficients belong to the space 
	of multiple zeta values of weight~$n$. We also explicitly compute these polynomials
	for all $n\le 14$.
\end{abstract}

\section{Introduction} 
Let $\Omega$ be a bounded connected domain with piecewise smooth boundary 
in $\RR^2$ and let us denote by 
$\Delta\coloneqq\frac{\dl^2}{\dl x^2}+\frac{\dl^2}{\dl y^2}$ 
the standard flat Laplace operator. When one considers suitable spaces of 
functions on $\ol{\Omega}$ with Dirichlet boundary conditions 
(i.e., vanishing on $\dl \Omega$) it is well known by the spectral theorem 
that~$\Delta$ possesses a discrete spectrum 
$\{\lambda_{k}(\Omega)\}_{k=1}^{\infty}$ of eigenvalues
\begin{equation*}
0<\lambda_{1}\le \lambda_{2}\le \lambda_{3}\le 
\dots \,
\end{equation*}
with corresponding finite-dimensional eigenspaces $\mathrm{Eig}(\lambda_{k})$ 
of smooth eigenfunctions. In other words each eigenfunction $\phi \in \mathrm{Eig}(\lambda_{k})$ satisfies the following boundary value PDE:
\begin{equation*}
\begin{cases}
\Delta \phi + \lambda_k\phi \= 0\,,\\
\phi|_{\dl \Omega} \= 0\,.
\end{cases}
\end{equation*}
One can consider many aspects regarding the asymptotics of Dirichlet 
eigenvalues and Dirichlet eigenfunctions, a large portion of which could be 
regarded as ``classical'' and being intensely studied: without aiming at being 
thorough, we just mention, for example, various forms of Weyl laws prescribing 
the asymptotics of large eigenvalues $ \lambda_k$ in terms of the underlying 
geometric data; concentration phenomena for eigenfunctions and level set 
distribution; behaviour of eigenvalues with respect to domain perturbation; 
etc. For a very thorough and accessible overview we refer to 
the treatments in \cite{Cha}, \cite{Ze}.

In this note we are interested in the behavior of the eigenvalues with 
respect to domain perturbations in the special case of regular polygons. 
Let $\Pp_N$ be a regular polygon of area $\pi$ with $N\ge3$ sides. 
We study the behavior of $\lambda_k(\Pp_N)$ as $N$ goes to
infinity. More precisely, we are interested in computing the coefficients 
$C_{k,n}$ of the asymptotic series
\begin{equation}\label{eq:asympt}
\frac{\lambda_k(\Pp_N)}{\lambda_k(\DD)} \es\sim 
1+\frac{C_{k,1}}{N}+\frac{C_{k,2}}{N^2}+\frac{C_{k,3}}{N^3}+\dots\,,
\end{equation}
where $\DD$ is the unit disk.

The above problem has been considered in several previous works. As an outcome 
of the works \cite{Mo}, \cite{GS1}, \cite{GS2} the first four 
coefficients $C_{1,i} $ were computed and, to a certain surprise, expressed 
as integer multiples of the Riemann zeta function. Roughly speaking, the 
corresponding methods (Calculus of Moving Surfaces) require one to consider 
an explicit deformation of the polygon into a disk and study the evolution 
of the corresponding eigenfunction.

Later on, the next two coefficients $C_{1,5}, C_{1,6}$ were also computed 
and given in a similar form (cf. \cite{Bo}). Recently, in \cite{Jo} 
an expression for the next two coefficients $C_{1,7}, C_{1,8}$ was 
proposed as a result of high-precision numerics and certain linear 
regression methods.

To summarize the above results, assuming that the polygons under consideration 
are normalized so that $\Area(\Pp_N) = \pi$, the asymptotic expansion 
(with proposed seventh and eighth terms being conjectural) is the following
\begin{equation} \label{eq:expansionsmall}
\begin{split}
\frac{\lambda_1(\Pp_N)}{\lambda_1} \= 
1&+ \frac{4 \zeta(3)}{N^3} 
+ \frac{(12 - 2 \lambda_1) \zeta(5)}{N^5} 
+ \frac{(8+4\lambda_1) \zeta^2(3)}{N^6} \\ 
&+ \frac{(36 - 12 \lambda_1 - \frac{1}{2}\lambda_1^2) \zeta(7)}{N^7} 
+ \frac{(48 + 8 \lambda_1 + 2 \lambda_1^2) \zeta(3) \zeta(5)}{N^8} 
+ O(N^{-9}) \,.
\end{split}
\end{equation}
Here $\lambda_1=\lambda_1(\DD)$ is the first Dirichlet eigenvalue of the unit 
disk. Recall that $\lambda_1=j_{0,1}^2$, where $j_{0,1}$ is the smallest 
positive zero of the Bessel function of the first kind $J_0$.

Formula~\eqref{eq:expansionsmall} might lead one to suspect that all higher 
coefficients of the asymptotic expansion can be expressed as polynomials in 
$\lambda_1$ with coefficients that are polynomials in odd zeta values 
$\zeta(2m+1)$, $m\ge1$. As we will show below, this is indeed the case for 
the first $10$ coefficients of the expansion, but, assuming some widely 
believed algebraic independence results, the $11$-th coefficient is no longer 
of this form.

As further motivation we note a couple of related problems. Drawing inspiration from the Faber-Krahn inequality and a conjecture of P\'olya and Szeg\"o (which states that among all $n$-gons with the same area, the regular $n$-gon has the smallest first Dirichlet eigenvalue), it was conjectured in \cite{AF} that for all $N \geq 3$ and $\Area(\Pp_N) = \pi $, the first Dirichlet eigenvalues are monotonically decreasing in $N$, i.e.,
\begin{equation*}
\lambda_1(\mathcal{P}_N) > \lambda_1(\mathcal{P}_{N+1})\,.
\end{equation*}
So far the monotonicity has been confirmed by numerical experiments 
(cf.~\cite{AF}). Note, however, that since $\zeta(3)>0$, 
equation~\eqref{eq:expansionsmall} implies this inequality for all sufficiently 
large $N$. For further results along the theme of Faber-Krahn and eigenvalue 
optimization via regular polygons under the presence of various constraints 
(in- and circumradius normalization, etc.) we refer to \cite{Ni}, \cite{So} 
and the references therein. 

A further intriguing application of the above eigenvalue asymptotics can be 
found in \cite{Oi}, where the Casimir energy of a scalar field on $\Pp_N$ (and 
further generalized to $\Pp_N \times \mathbb{R}^k$) has been studied. For 
background we refer to \cite{Oi} and the accompanying references.

Before describing our main results, we briefly recall the definition of 
multiple zeta values. Multiple zeta values (MZVs) are real numbers defined by
	\[\zeta(m_1,\dots,m_r) \es{\coloneqq} \sum_{0<n_1<n_2<\dots<n_r}
	\frac{1}{n_1^{m_1}n_2^{m_2}\dots n_r^{m_r}}\,,\]
where $m_1,\dots,m_r$ are positive integers and $m_r>1$.
A multiple zeta value $\zeta(m_1,\dots,m_r)$ is said to have weight $n$
if $m_1+\dots+m_r=n$. We denote the $\QQ$-linear span of all multiple zeta 
values of weight~$n$ by~$\MZV_n$. Our main result is the following theorem.
\begin{theorem}
	\label{thm:mainthm}
	There exists a sequence of polynomials $C_n\in \MZV_{n}[\lambda]$, $n\ge1$,
	where $\MZV_n$ is the space of multiple zeta values of weight $n$, 
	such that 
	\begin{equation} \label{eq:mainid}
	\frac{\lambda_k(\Pp_N)}{\lambda_k} \es\sim
	1+\sum_{n=1}^{\infty}\frac{C_n(\lambda_k)}{N^n}
	\end{equation}
	whenever $\lambda_k$ is a radially-symmetric Dirichlet eigenvalue 
	of the unit disk.
\end{theorem}
Here radially-symmetric eigenvalues are $\lambda_k$'s for which 
the corresponding eigenfunction is radially-symmetric. Explicitly, the theorem 
applies whenever $\lambda_k=j_{0,m}^2$, where $j_{0,m}$ is the $m$-th root 
of the Bessel function $J_0(x)$. In particular, the theorem applies
in the case~$k=1$.

Our proof of Theorem~\ref{thm:mainthm} is based on an asymptotic version of the 
``method of particular solutions'' (see, e.g.,~\cite{FoHeMo},~\cite{MoPa}) 
and it provides an explicit procedure for producing increasingly better 
approximations (at least when $N\to\infty$) to both the eigenvalues and the 
eigenfunctions (as long as they correspond to the radially symmetric 
eigenfunctions on the unit disk).
We find that not only the eigenvalues, but the eigenfunctions themselves
have asymptotic expansions in powers of $1/N$ with interesting coefficients
(these coefficients turn out to be multiple polylogarithms, see 
Table~\ref{tab:vn}).

\begin{table}[h!]
	{\def\arraystretch{1.5}
		\begin{tabularx}{\linewidth}{r |R}
			$n$ & $C_n(\lambda)$\\
			\hline
			$1$   & $0$\\
			$2$   & $0$\\
			$3$   & $4\Z_3$\\
			$4$   & $0$ \\
			$5$   & $-2\Z_{5}\lambda + 12\Z_{5}$\\
			$6$   & $4\Z_{3}^2\lambda + 8\Z_{3}^2$\\
			$7$   & $-\tfrac{1}{2}\Z_{7}\lambda^2 - 12\Z_{7}\lambda + 36\Z_{7}$\\
			$8$   & $2\Z_{5}\Z_{3}\lambda^2 + 8\Z_{5}\Z_{3}\lambda + 48\Z_{5}\Z_{3}$\\
			$9$   & $-\tfrac{1}{4}\Z_{9}\lambda^3 - \tfrac{104}{9}\Z_{9}\lambda^2
			+(-\tfrac{146}{3}\Z_{9}+\tfrac{80}{3}\Z_3^3)\lambda
			+(\tfrac{340}{3}\Z_{9} + \tfrac{32}{3}\Z_3^3)$\\
			$10$   & $(\Z_{7}\*\Z_{3} + \Z_{5}^2)\lambda^3 
			+(39\Z_{7}\Z_{3} -6\*\Z_{5}^2)\lambda^2
			+(-24\Z_{7}\Z_{3} - 12\Z_{5}^2)\lambda 
			+(144\Z_{7}\Z_{3} + 72\Z_{5}^2)$\\
			$11$   &{\footnotesize $-\tfrac{5}{32}\Z_{11}\lambda^4 
				+(-\tfrac{661}{60}\Z_{11}+\tfrac{1}{5}\Z_{3,5,3}^{\mathrm{sv}})\lambda^3
				+(-\tfrac{1623}{20}\Z_{11} + 80\Z_5\Z_3^2 + 
				\tfrac{54}{5}\Z_{3,5,3}^{\mathrm{sv}})\lambda^2 
				+(-176\Z_{11} + 176\Z_5\Z_3^2)\lambda$ $+(372\Z_{11} + 96\Z_5\Z_3^2)$}\\
			$12$   &{\footnotesize $(\tfrac{5}{8}\*\Z_{9}\*\Z_{3} + \tfrac{11}{8}\Z_{7}\Z_{5})\lambda^4 
				+(\tfrac{107}{3}\Z_{9}\Z_{3} + \tfrac{47}{2}\Z_{7}\Z_{5})\lambda^3
				+(456\Z_{9}\Z_{3} - 207\Z_{7}\Z_{5} - 16\Z_{3}^4)\lambda^2$
				$+(-\tfrac{488}{3}\Z_{9}\Z_{3}-216\Z_{7}\Z_{5}
				+\tfrac{272}{3}\Z_{3}^4)\lambda
				+(\tfrac{1360}{3}\Z_{9}\Z_{3}+432\Z_{7}\*\Z_{5}+\tfrac{32}{3}\Z_{3}^4)$}\\
			$13$   & {\footnotesize $-\tfrac{7}{64}\Z_{13}\lambda^5 + (-\tfrac{226501}{16800}\Z_{13}+\Z_7\Z_3^2 - \tfrac{31}{10}\Z_5^2\Z_3  - \tfrac{157}{1400}\Z_{5,3,5}^{\mathrm{sv}}+ \tfrac{5}{56}\Z_{3,7,3}^{\mathrm{sv}})\lambda^4 
				+ (-\tfrac{1283839}{8400}\Z_{13} + 34\Z_7\Z_3^2 $ $+ \tfrac{256}{5}\Z_5^2\Z_3-\tfrac{549}{350}\Z_{5,3,5}^{\mathrm{sv}} + \tfrac{59}{28}\Z_{3,7,3}^{\mathrm{sv}})\lambda^3 
				+ (-\tfrac{1447393}{1400}\Z_{13} + 1236\Z_7\Z_3^2 - \tfrac{1128}{5}\Z_5^2\Z_3 - \tfrac{12339}{175}\Z_{5,3,5}^{\mathrm{sv}} + \tfrac{747}{14}\Z_{3,7,3}^{\mathrm{sv}})\lambda^2$
				$+ (-618\Z_{13}+336\Z_7\Z_3^2 + 336\Z_5^2\Z_3)\lambda + (1260\Z_{13} + 288\Z_7\Z_3^2 + 288\Z_5^2\Z_3)$}\\
			$14$   & {\footnotesize $
				(\tfrac{7}{16}\Z_{11}\Z_{3}
				+ \Z_{9}\Z_{5}
				+ \tfrac{9}{16}\Z_{7}^2) \lambda^5
				+ (\tfrac{10169}{240}\Z_{11} \Z_{3}
				+ \tfrac{467}{8}\Z_{9}\Z_{5}
				+ \tfrac{175}{16}\Z_{7}^2
				- \tfrac{1}{5}\Z_{3,5,3}^{\mathrm{sv}}\Z_{3}) \lambda^4 
				+ (\tfrac{20381}{30}\Z_{11} \Z_{3}
				+ \tfrac{1300}{9}\Z_{9}\Z_{5}
				+ \tfrac{483}{4}\Z_{7}^2
				+ 40\Z_{5}\Z_{3}^3
				+ \tfrac{28}{5}\Z_{3,5,3}^{\mathrm{sv}}\Z_{3}) \lambda^3 
				+ (\tfrac{32902}{5}\Z_{11} \Z_{3}
				- \tfrac{7306}{3}\Z_{9}\Z_{5}
				- \tfrac{3627}{2}\Z_{7}^2
				+ \tfrac{3664}{3}\Z_{5}\Z_{3}^3
				+ \tfrac{1296}{5}\Z_{3,5,3}^{\mathrm{sv}}\Z_{3}) \lambda^2 
				+ (- 664\Z_{11}\Z_{3}
				- \tfrac{2824}{3}\Z_{9}\Z_{5}
				- 540\Z_{7}^2
				+ \tfrac{2752}{3}\Z_{5}\Z_{3}^3) \lambda
				+ (1488\Z_{11}\Z_{3}
				+ 1360\Z_{9}\Z_{5}
				+ 648\Z_{7}^2
				+ 128\Z_{5}\Z_{3}^3 )$}\\
			\hline
	\end{tabularx}}
	\smallskip
	\caption{Coefficients of the asymptotic expansion for $n\le 14$}
	\label{tab:cn}
\end{table}

Our approach gives an explicit symbolic algorithm for computing 
the polynomials $C_n$. We have calculated $C_n$ 
for $n\le 12$ using an implementation of this algorithm in the computer algebra 
system SAGE~\cite{SAGE} and for $n\le 14$ using an optimized parallel 
implementation in Julia~\cite{Julia}. As an immediate corollary we confirm 
the~$7$-th and~$8$-th terms in the asymptotic expansion~\eqref{eq:asympt} that 
were conjectured in~\cite{Jo}. We collect the results of our calculations in 
Table~\ref{tab:cn}. The initial expressions in terms of MZVs that we get by 
directly applying our algorithm are rather unwieldy and to obtain the simpler 
expressions given in the table we have used the MZV Datamine~\cite{BBV}.
In the table we use the notation $\Z_n\coloneqq\zeta(n)$ and
\begin{align*}
  	\Z_{3,5,3}^{\mathrm{sv}} &\coloneqq 2\zeta(3,5,3)-2\zeta(3)\zeta(3,5)-10\zeta(3)^2\zeta(5) \,,\\
  	\Z_{5,3,5}^{\mathrm{sv}} &\coloneqq 
  	2\zeta(5,3,5)-22\zeta(5)\zeta(3,5)-120\zeta(5)^2\zeta(3)-10\zeta(5)\zeta(8) \,,\\
  	\Z_{3,7,3}^{\mathrm{sv}} &\coloneqq 2\zeta(3,7,3)-2\zeta(3)\zeta(3,7)-28\zeta(3)^2\zeta(7)-24\zeta(5)\zeta(3,5) -144\zeta(5)^2\zeta(3)-12\zeta(5)\zeta(8)
\end{align*}
for the first few nontrivial single-valued MZVs.
The space of single-valued multiple zeta values of weight~$n$ is an 
important subspace of~$\MZV_n$ that was introduced by Brown~\cite{Br2}.
Single-valued MZVs appear, for example, in computation of 
string amplitudes, and as coefficients of Deligne's associator (for other 
examples, see the references in~\cite{Br2}).

\begin{conjecture} \label{conj:singlevalued}
   	The polynomial $C_n(\lambda)$ belongs to $\MZV_{n}^{\mathrm{sv}}[\lambda]$
   	for all $n\ge 1$, where $\MZV_{n}^{\mathrm{sv}}$ denotes the space of
   	single-valued multiple zeta values of weight~$n$.
\end{conjecture}
The results given in Table~\ref{tab:cn} confirm this conjecture for $n\le 14$, 
and in~\cite{BJMR} we give strong numerical evidence in its support 
also for $n=15$ and $n=16$. If true, Conjecture~\ref{conj:singlevalued} 
would also explain the curious fact that when $\Pp_N$ is normalized to have 
area~$\pi$ (as opposed to, say, normalizing $\Pp_N$ to have 
circumradius~1), the low order coefficients of the resulting asymptotic 
expansion do not involve even zeta values (see~\cite[p.~125]{Bo}).

By analyzing the general recursion for $C_n(\lambda)$ obtained in the proof of Theorem~\ref{thm:mainthm} we obtain a formula for the generating function 
of the first two coefficients of the polynomials $C_n(\lambda)$.
\begin{theorem}\label{thm:mainthm2}
	The coefficients $C_n(0)$ and $C_n'(0)$ satisfy the generating 
	series identity
	\begin{equation}\label{eq:constterm}
	\sum_{n\ge 0}(C_n(0)+C_n'(0)\lambda)z^n \= 
	\frac{\Gamma(1+z)^2\Gamma(1-2z)}{\Gamma(1-z)^2\Gamma(1+2z)}
	\Big(1-\frac{\lambda}{2}\sum_{n\ge1}\frac{(2z)_n^2z^3}{n!^2(z+n)^3}\Big)\,,
	\end{equation}
	where $(x)_n=x(x+1)\dots (x+n-1)$ denotes the rising Pochhammer symbol.
\end{theorem}
Using this formula we also prove the following.
\begin{theorem} \label{thm:oddzeta}
	The coefficients $C_n(0)$ and $C_n'(0)$ are polynomials with rational coefficients in odd zeta values of homogeneous weight~$n$.
\end{theorem}
The proof is based on a hypergeometric identity~\eqref{eq:trigamma} 
of Ramanujan-Dougall type that could be of independent interest.
Theorem~\ref{thm:oddzeta} confirms Conjecture~\ref{conj:singlevalued} 
for the first two coefficients $C_n(0)$ and $C_n'(0)$. 
Note that the expression for $C_{11}$ from Table~\ref{tab:cn} shows that 
$C_{11}''(0)$ involves $\Z_{3,5,3}^{\mathrm{sv}}$, and 
thus the claim of Theorem~\ref{thm:oddzeta} in general fails for $C_{n}''(0)$, 
assuming the (widely believed) algebraic independence of 
$\Z_{3,5,3}^{\mathrm{sv}}$ and $\zeta(2m+1)$, $m\ge 1$ (see~\cite[p.~35]{Br2}).

As a final remark, we note that the normalizing factor
	\[\frac{\Gamma(1+z)^2\Gamma(1-2z)}{\Gamma(1-z)^2\Gamma(1+2z)}\]
that appears in several of our formulas is a specialization of Virasoro's
closed bosonic string amplitude~\cite{Vi}. We do not know if this is a simple
coincidence, or if there is some conceptual explanation for this.

\smallskip
\section{Multiple polylogarithms and multiple zeta values}
In this section we very briefly recall some basic properties of multiple polylogarithms and multiple zeta values. For a much more detailed introduction (including the algebraic structure and interpretation of MZVs as periods of mixed Tate motives) we refer the reader to~\cite{W},~\cite{BF}.

Let $m_1,\dots,m_r$ be positive integers.
The one-variable multiple polylogarithm $\Li_{m_1,\dots,m_r}(z)$ is an analytic 
function defined by the power series
	\begin{equation} \label{eq:limdef}
	\Li_{m_1,\dots,m_r}(z) \es{\coloneqq} \sum_{0<n_1<n_2<\dots<n_r}
	\frac{z^{n_r}}{n_1^{m_1}n_2^{m_2}\dots n_r^{m_r}}\,,\qquad |z|<1\,.
	\end{equation}
For $m_r>1$ the above series converges absolutely for $|z|\le 1$ and we define
	\begin{equation} \label{eq:mzvdef}
	\zeta(m_1,\dots,m_r) \es{\coloneqq} \sum_{0<n_1<n_2<\dots<n_r}
	\frac{1}{n_1^{m_1}n_2^{m_2}\dots n_r^{m_r}} \= \Li_{m_1,\dots,m_r}(1)\,.
	\end{equation}

The numbers $\zeta(m_1,\dots,m_r)$ are called multiple zeta values (MZVs).
We denote by $\MZV$ the $\QQ$-linear span of all multiple zeta values
and by $\MZV_k$ the $\QQ$-linear span of all multiple zeta values of weight $k$,
i.e., the linear span of $\zeta(m_1,\dots,m_r)$ over all $r$-tuples 
$(m_1,\dots,m_r)$ satisfying $m_1+\dots+m_r=k$. 
The $\QQ$-vector space $\MZV$ forms an algebra (see~\eqref{eq:shuffle} below),
and multiplication respects weight, i.e., 
$\MZV_k \cdot \MZV_l \subseteq \MZV_{k+l}$.
Zagier~\cite{Za} has conjectured that there are no rational linear relations 
between elements of $\MZV_k$ for different $k$ 
(that is, that $\MZV=\bigoplus_{k\ge0} \MZV_k$) 
and that $\dim \MZV_k=d_k$, where $d_k$ are defined by the generating series
\[\frac{1}{1-x-x^3} \= \sum_{k\ge0} d_kx^k\,.\]
The upper bound $\dim \MZV_k \le d_k$ has been proved independently by 
Goncharov and Terasoma, but no nontrivial lower bounds for $\dim \MZV_k$ are presently known.

For our purposes it is more convenient to index multiple polylogarithms 
by words in two letters $X=\{x_0,x_1\}$, reflecting their
structure as iterated integrals as opposed to the definition as an infinite 
sum~\eqref{eq:limdef}. In our treatement we mainly follow Brown~\cite{Br1} (see also~\cite{NPH}).
Let $X^{\times}$ be the free noncommutative monoid generated by~$X$, 
i.e., the set of all words in $x_0,x_1$ equipped with the concatenation 
product. Then $\{\Li_{w}\}_w$ is a family of analytic functions on 
the cut plane $\CC\sm ((-\infty,0]\cup [1,\infty))$ satisfying the recursive 
relations
\begin{equation} \label{eq:polylogrec}
\frac{d}{dz}\Li_{x_0w}(z)\= \frac{\Li_{w}(z)}{z}\,,
\qquad
\frac{d}{dz}\Li_{x_1w}(z)\= \frac{\Li_{w}(z)}{1-z}\,,
\qquad w\in X^{\times}
\end{equation}
together with the following initial conditions: $\Li_{e}(z)=1$, 
$\Li_{x_0^n}(z)=\frac{1}{n!}\log^n(z)$, and $\lim_{z\to 0}\Li_{w}(z)=0$ 
for all $w\in X^{\times}$ not of the form $x_0^n$. 
Here $e\in X^{\times}$ denotes the empty word.
These conditions uniquely determine $\Li_{w}(z)$ and for 
$m_1,\dots,m_r\ge 1$ one has
\begin{equation} \label{eq:polylogrel}
\Li_{x_0^{m_r-1}x_1\dots x_{0}^{m_1-1}x_1}(z) 
\= \Li_{m_1,\dots,m_r}(z)\,.
\end{equation}
We also extend the notation $\Li_w(z)$ by linearity to the elements of the 
monoid ring $\CC\langle X\rangle$, i.e., for any formal combination 
$\sum_{i} a_iw_i$ we set $\Li_{\sum_{i}a_iw_i}(z) = \sum_{i}a_i\Li_{w_i}(z)$.
Note that the algebra $\CC\ang{X}$ is graded by word 
length, and we denote by $\CC\ang{X}_n$ the $n$-the graded piece.
We will also write $|w|$ for the length of $w\in X^{\times}$ and we will
say that the function $\Li_{w}(z)$ has weight $|w|$.
(Since the functions $\Li_w(z)$ are linearly independent over $\CC(z)$, see~\cite{Br1}, this notion of weight is well-defined.)

An important property of the space of multiple polylogarithms is that
it is closed under mutliplication. More precisely, one has
\begin{equation} \label{eq:shuffle}
\Li_{w}(z)\Li_{w'}(z) \= \Li_{w\shuffle w'}(z)\,.
\end{equation}
Here $\shuffle\colon\CC\ang{X}\times\CC\ang{X}\to\CC\ang{X}$ 
denotes the shuffle product that is defined on words by
\[a_1\dots a_k \shuffle a_{k+1}\dots a_{k+l}
\= \sum_{\sigma} a_{\sigma(1)}\dots a_{\sigma(k+l)}\,,\]
where $\sigma$ runs over all permutations satisfying 
$\sigma^{-1}(1)<\dots<\sigma^{-1}(k)$
and $\sigma^{-1}(k+1)<\dots<\sigma^{-1}(k+l)$.

Note that for $w\in X^{\times}x_1$ the function $\Li_w(z)$ extends analytically 
to $\CC\sm [1,\infty)$, and for $w\in x_0X^{\times}x_1$ it is moreover 
continuous on $\ol{\DD}$. We will call the words $w\in x_0X^{\times}x_1$ 
\textit{convergent} and we will also call convergent any formal linear 
combination of convergent words in $\CC\ang{X}$ (in other words, all elements 
of $x_0\CC\ang{X}x_1$ are convergent). An important corollary 
of~\eqref{eq:shuffle}, is that for any $w\in X^{\times}x_1$ there exists a 
unique collection of convergent elements 
$w_0,w_1,\dots,w_k\in \QQ\langle X\rangle$ such that
\[\Li_w(z) \= 
\Li_{w_0}(z)+\Li_{w_1}(z)\Li_1(z)+\dots+\Li_{w_k}(z)\Li_1^k(z)\,.\]
Recall that $\Li_1(z)=-\log(1-z)$. This allows one to define the multiple 
zeta value $\zeta(w)=\Li_{w}(1)$ in cases when the series diverges by setting 
$\Li_{w}(1)\coloneqq\Li_{w_0}(1)$ for all $w\in X^{\times}x_1$. 
%(More generally, one can define the regularized values $\zeta(w)$ for all $w\in X^{\times}$, but we will not need this.)

First, let us note the following simple corollaries of the definition of $\Li_w$. 
\begin{lemma} \label{lem:convolution}
	For all $w\in X^{\times}x_1$ and $k\ge0$ we have
	\[\Li_{x_0^{k+1}w}(z) \= \frac{1}{k!}\int_{0}^{z}\Li_w(t)\log^k(z/t)\frac{dt}{t}\,, 
	\qquad z\in \ol{\DD}\,.\]
\end{lemma}
\begin{proof}
	This follows trivially from~\eqref{eq:polylogrec} by induction on $k$.
\end{proof}

\begin{lemma} \label{lem:average}
	For all $w\in X^{\times}x_1$ we have
	\[\int_{0}^{1}\Li_{w}(e^{2\pi i t})dt \= 0\,.\]
\end{lemma}
\begin{proof}
	This follows from $\Li_{w}(e^{2\pi it}) \= \frac{1}{2\pi i}\frac{d}{dt}\Li_{x_0w}(e^{2\pi i t})$.
\end{proof}

Our proof of Theorem~\ref{thm:mainthm} is based on the following simple result.
\begin{proposition} \label{prop:polylogproduct}
	Let $u,v\in X^{\times}x_1$ and let $k=|u|+|v|$. Then there 
	exist elements $\alpha(u,v)$ and $\beta(u,v)$ in $\bigoplus_{j=0}^{k-1}\MZV_{j}\otimes \QQ\ang{X}_{k-j-1}x_1$
	and $A_{u,v}\in\MZV_k$ such that 
	\begin{equation} \label{eq:polylogproduct}
	\Li_{u}(z)\Li_{v}(z^{-1}) \= A_{u,v}
	+\Li_{\alpha(u,v)}(z)+\Li_{\beta(u,v)}(z^{-1})
	\,,\qquad z\in \CC\sm [0,+\infty)\,.
	\end{equation}
	Moreover, if $u$ and $v$ are convergent, then 
	$\alpha(u,v)$ and $\beta(u,v)$ are also convergent.
\end{proposition}
\begin{proof}
	We will prove the statement by induction on $k$ for all 
	$u,v\in X^{\times}x_1\sqcup \{e\}$.
	For the base of induction, when either $u=e$ or $v=e$ the identity
	becomes trivial if we set $A_{u,e}=A_{e,v}=0$, $\alpha(u,e)=u$,
	$\alpha(e,v)=0$, and $\beta(u,e)=0$, $\beta(e,v)=v$.
	
	Note that from~\eqref{eq:polylogrec} it follows that
	for all $w\in X^{\times}$ we have
	\begin{equation} \label{eq:polylogrec2}
	\frac{d}{dz}\Li_{x_0w}(z^{-1})\= -\tfrac{1}{z}\Li_{w}(z^{-1})\,,
	\qquad
	\frac{d}{dz}\Li_{x_1w}(z^{-1})\= (\tfrac{1}{z}+\tfrac{1}{1-z})\Li_{w}(z^{-1})\,.
	\end{equation}
	Therefore, if we set
	\[F_{u,v}(z) \= \Li_{u}(z)\Li_{v}(z^{-1})\,,\]
	then for any $a,b\in X$, $u,v\in X^{\times}$ we have
	\begin{align*}
	&\frac{d}{dz}F_{au,bv}(z) 
	\= \phi_a(z)F_{u,bv}(z)+\psi_b(z)F_{au,v}(z) \\
	&\= \phi_a(z)
	(A_{u,bv}+\Li_{\alpha(u,bv)}(z)+\Li_{\beta(u,bv)}(\tfrac{1}{z}))
	+\psi_b(z)(A_{au,v}+\Li_{\alpha(au,v)}(z)+\Li_{\beta(au,v)}(\tfrac{1}{z}))
	\,,
	\end{align*}
	where $\phi_{x_0}(z)=\frac{1}{z}$, $\phi_{x_1}(z)=\frac{1}{1-z}$,
	$\psi_{x_0}(z)=-\frac{1}{z}$, and $\psi_{x_1}(z)=\frac{1}{z}+\frac{1}{1-z}$.
	In view of this we recursively define
	\begin{equation} \label{eq:alphabetarecursion}
	\begin{split}
	\alpha(au,bv) &\= a\alpha(u,bv)+\tilde{b}\alpha(au,v)
	+(A_{u,bv}\delta(a)+A_{au,v}\tfrac{1-\delta(b)}{2})x_1 \\
	\beta(au,bv) &\= \tilde{a}\beta(u,bv)+b\beta(au,v)
	+(A_{au,v}\delta(b)+A_{u,bv}\tfrac{1-\delta(a)}{2})x_1\,,
	\end{split}
	\end{equation}
	where $\widetilde{x_0}=-x_0$, $\widetilde{x_1}=x_0+x_1$, and
	$\delta$ is defined by $\delta(x_1)=1$, $\delta(x_0)=-1$.
	Then by induction we obtain that
	\[\Li_{u}(z)\Li_{v}(z^{-1})-\Li_{\alpha(u,v)}(z)-\Li_{\beta(u,v)}(z^{-1})
	\= {\rm const} \es{=:} A_{u,v}\,.\]
	
	If $u,v,\alpha$, and $\beta$ are all convergent, then we may simply take
	$z=1$ to get 
	$A_{u,v}=\Li_u(1)\Li_v(1)-\Li_{\alpha(u,v)}(1)-\Li_{\beta(u,v)}(1)$. 
	Otherwise we take $z=e^{2\pi i x}$ and take the limit 
	$x\to 0+$, which corresponds to a regularization of $\Li_w(1)$ given by
	\[\Li_w(e^{\pm 2\pi i 0}) \coloneqq \sum_{j=0}^{k}\Li_{w_j}(1)(\pm\tfrac{\pi i}{2})^j\,,\]
	where $\Li_w(z) = \sum_{j=0}^{k}\Li_{w_j}(z)\Li_1^j(z)$
	with all $w_j$ in $x_0\CC\ang{X}x_1$.
	Thus $A_{u,v}\in\MZV_k$ and by induction we get also that 
	$\alpha(u,v)$ and $\beta(u,v)$ belong to $\bigoplus_{j=0}^{k-1}\MZV_{j}\otimes \QQ\ang{X}_{k-j-1}x_1$.
	
	To verify the last claim let us consider $u=x_0u'$, $v=x_0v'$.
	The recursive definition~\eqref{eq:alphabetarecursion} with $a=b=x_0$
	shows that $\alpha(u,v)$ and $\beta(u,v)$ would be convergent if we can
	show that $A_{u',x_0v'}=A_{x_0u',v'}$. But the calculation of
	the derivative of $F_{u,v}$ shows that
	\begin{align*}
	&i\frac{d}{dt}F_{u,v}(e^{it}) 
	\= A_{u',x_0v'}-A_{x_0u',v'}+\Li_{w}(e^{it})+\Li_{w'}(e^{-it})
	\end{align*}
	for some $w,w'\in \CC\ang{X}x_1$ and hence, by Lemma~\ref{lem:average},
	we obtain $A_{u',x_0v'}=A_{x_0u',v'}$.
\end{proof}

\begin{remark}
	The proof shows that we may take $\beta(u,v)=\alpha(v,u)$.
	Note also that if we extend the definition of $\alpha(u,v)$, 
	$\beta(u,v)$, and $A_{u,v}$ to bilinear functionals on 
	$\CC\ang{X}x_1\times \CC\ang{X}x_1$, the identity~\eqref{eq:polylogproduct} remains true for all $u,v\in\CC\ang{X}x_1$.	
\end{remark}

As a corollary of the above proposition we have the following curious fact.

\begin{corollary}
	For all $u,v\in X^{\times}x_1$ we have
	\[\int_{0}^{1}\Li_{u}(e^{2\pi i t})\Li_{v}(e^{-2\pi i t})dt
	\;\;\in\;\; \MZV_{|u|+|v|}\,.\]
\end{corollary}
\begin{proof}
	It follows from~\eqref{eq:polylogproduct}
	and Lemma~\ref{lem:average} that
	\begin{equation*} \label{eq:cuvinnerproduct}
	A_{u,v} \= \int_{0}^{1}\Li_{u}(e^{2\pi i t})\Li_{v}(e^{-2\pi i t})dt\,,
	\end{equation*}
	and the claim then follows from Proposition~\ref{prop:polylogproduct}.
\end{proof}

As a further corollary, note that when $u$ and $v$ are both convergent 
equation~\eqref{eq:polylogproduct} implies
	\[
	\re \Li_{u}(z)\Li_{v}(\ol{z}) \= A_{u,v}
	+\re \Li_{\alpha(u,v)+\beta(u,v)}(z)\,,\qquad |z|=1\,.
	\]
This formula thus gives a purely algebraic solution of the Dirichlet boundary 
value problem $u(z)=\phi(z)$ for $|z|=1$ where $\phi$ is of the form 
$\phi(z)=\re \Li_{u}(z)\Li_{v}(\ol{z})$ and $u$ is sought to be harmonic
in~$\DD$. It is exactly in this form that we will
use Proposition~\ref{prop:polylogproduct} in the proof of Theorem~\ref{thm:mainthm}.

\smallskip
\section{Proof of Theorem~\ref{thm:mainthm}}
Let us fix the polygon $\Pp_N\subset \CC$ to be the convex hull of 
$\{c\zeta^j\}_{0\le j<N}$, where $\zeta$ is a primitive $n$-th root of unity, 
and $c>0$ is chosen so that $\Area(\Pp_N) = \pi$. 
We will utilize the classical Schwarz-Christoffel map,
$f\colon\mathbb{D}\to\Pp_N$, which maps the unit disk~$\DD$ conformally 
onto $\Pp_N$. It is given by any of the following equivalent expressions
	\begin{equation}\label{eq:confmap}
	f(z) \= c_Nz\, {}_{2}F_{1}\Big(\frac{2}{N},\frac{1}{N},1+\frac{1}{N};z^N\Big) 
	\= c_N\int_{0}^{z}\frac{d\zeta}{(1-\zeta^N)^{2/N}}
	%\=\frac{c_N}{N}\int_{0}^{z}\frac{d\zeta^N}{\zeta^{N-1}(1-\zeta^N)^{2/N}}\,,
	\end{equation}
where the constant
	\begin{equation} \label{eq:schwarzconst}
	c_N \=
	\sqrt{\frac{\Gamma(1-1/N)^2\Gamma(1+2/N)}{\Gamma(1+1/N)^2\Gamma(1-2/N)}}
	\end{equation}
is determined by the condition that $\Area(\Pp_N) = \pi$.
Here ${}_{2}F_{1}$ is the ordinary Gauss hypergeometric function
	\[{}_{2}F_{1}(a,b,c;z) \= 
	\sum_{n\ge0}\frac{(a)_n(b)_n}{(c)_n}\frac{z^n}{n!}\,,
	\qquad |z|<1\,,\]
where $(x)_n\coloneqq x(x+1)\dots(x+n-1)$ denotes the rising Pochhammer symbol.

The first fact that we will need is that the function 
$F_N(x) \coloneqq {}_{2}F_{1}(2/N,1/N,1+1/N;x)$ can be 
expanded as a power series in~$1/N$ (convergent for $N\ge 3$) 
whose coefficients are multiple polylogarithms.
Let us recall the definition of Nielsen polylogarithms 
(see~\cite{Ko}, \cite{CGR})
	\begin{equation}\label{eq:nielsendef}
	S_{n,p}(z) \= \frac{(-1)^{n+p-1}}{(n-1)!p!}
	\int_{0}^{1}\log^{n-1}(t)\log^{p}(1-zt)\frac{dt}{t}
	\= \Li_{1,\dots,1,n+1}(z)
	\= \Li_{x_0^nx_1^p}(z)\,.
	\end{equation}    
\begin{lemma}
	For all $N\ge3$ and $|x|\le 1$ we have
	\begin{equation} \label{eq:hypergexp}
	{}_{2}F_{1}\Big(\frac{2}{N},\frac{1}{N},1+\frac{1}{N};x\Big)
	\= 1+\sum_{n=2}^{\infty}S_n(x)N^{-n} \,,
	\end{equation}
	where
	\[S_{n}(x) \= \sum_{j=1}^{n-1}(-1)^{j-1}2^{n-j}S_{j,n-j}(x)\,.\]
\end{lemma}
\begin{proof}
	This follows by expanding in powers of $1/N$ the right hand side of
	\[{}_{2}F_{1}\Big(\frac{2}{N},\frac{1}{N},1+\frac{1}{N};x\Big)
	\= 1+\frac{1}{N}\int_{0}^{1}t^{1/N}((1-tx)^{-2/N}-1)\frac{dt}{t}\]
	and using the definition~\eqref{eq:nielsendef}.
\end{proof}

We will also need a formula for the asymptotic expansion of the Bessel 
function $J_0(x)$ around its zero. Recall that $J_0(x)$ satisfies 
$xJ_0''(x)+J_0'(x)+xJ_0(x)=0$ and
can be defined by the Taylor series
	\[J_0(x) \= \sum_{n\ge 0}\frac{(-\frac{x^2}{4})^n}{n!^2}\,.\]
\begin{proposition} \label{prop:besselasymp}
	(i) For $n\ge 0$ define
	\[E_n(x) \coloneqq \sum_{j=0}^{n}\frac{e^{2jx}(x+H_{n-j}-H_j)}{j!^2(n-j)!^2}\,,\]
	where $H_n=\sum_{j=1}^{n}\frac{1}{j}$, $H_0=0$.
	Then $E_n(x) = O(x^{2n+1})$, $x\to 0$.
	
	(ii)
	Let $\alpha$ be a zero of $J_0(x)$. Then
	\begin{equation} \label{eq:besselexp}
	-\frac{J_0(\alpha e^{x})}{\alpha J_1(\alpha)} \= 
	\sum_{n\ge 0}(-\tfrac{\alpha^2}{4})^n
	E_n(x)\,.
	\end{equation}
\end{proposition}
\begin{proof}
	(i) Define $W_n$ by 
	\[W_n(x)\coloneqq\int_{0}^{\infty}E_n(xt)e^{-t}dt
	\= \sum_{j=0}^{n}\frac{1}{j!^2(n-j)!^2}\Big(\frac{x}{(1-2jx)^2}+
	\frac{H_{n-j}-H_j}{1-2jx}\Big)\,.\]
	Note that if $E_n(x)=\sum_{m}a_mx^m$, then $W_n(x)=\sum_{m}m!a_mx^m$, 
	so it suffices to show that $W_n(x)=O(x^{2n+1})$, $x\to 0$. 
	We claim that
	\[W_n(x) \= \frac{4^nx^{2n+1}}{\prod_{j=1}^{n}(1-2jx)^2}\,,\] 
	which clearly implies $W_n(x)=O(x^{2n+1})$. Let $R_n(x)$ denote the difference between the left hand side and the right hand side in the 
	above equation. A simple calculation shows that $R_n(x)\to 0$ 
	as $x\to\infty$, and since $R_n\in\QQ(x)$, 
	it is enough to show that it has no poles. 
	The only potential singularities are at $x=\frac{1}{2k}$, $1\le k \le n$. If we let $x=\frac{1}{2k}-\eps$, then
	\[n!^2R_n(x) \= \binom{n}{k}^2\Big(\frac{\frac{1}{2k}-\eps}{(2k\eps)^2}+\frac{H_{n-k}-H_k}{2k\eps}\Big)-\frac{\eps^{-2}(\frac{1}{2k}-\eps)^{2n+1}}{\prod_{j\ne k}(\frac{1}{2j}-\frac{1}{2k}+\eps)^2}+O(1)\,,\qquad \eps\to 0\,.\]
	(Here $\prod_{j\ne k}$ denotes the product over $1\le j\le n$, $j\ne k$.) Then the coefficient in front $\eps^{-2}$ vanishes since
	$\binom{n}{k}^2=\prod_{j\ne k}(\frac{k}{j}-1)^{-2}$, and for $\eps^{-1}$
	the vanishing is equivalent to
	\[1-2k(H_{n-k}-H_k)
	\= (2n+1)+\sum_{\substack{1\le j\le n\\ j\ne k}}\frac{2j}{k-j}\,,\]
	which is again easy to verify.
	
	(ii) Let us denote the left hand side of~\eqref{eq:besselexp} 
	by $f(x)$ and the right hand side by~$g(x)$.
	From the differential equation
	$xJ_0''(x)+J_0'(x)+xJ_0(x)$ together with $J_0'(x)=-J_1(x)$, we obtain 
	that $f''(x)+\alpha^2e^{2x}f(x)=0$ and $f(0)=0$, $f'(0)=1$. 
	Thus, it is enough to check that $g(x)$ satisfies the same differential
	equation and initial conditions.  The conditions
	$g(0)=0$, $g'(0)=1$ follow from part~(i). 
	Using the easily checked identity
	\[\frac{e^{2jx}j^2(x+H_{n+1-j}-H_j)+e^{2jx}j}{j!^2(n+1-j)!^2}
	\=\frac{e^{2jx}j^2(x+H_{n-j+1}-H_{j-1})}{(j-1)!^2(n-j+1)!^2}
	\,,\qquad 1\le j\le n+1\]
	we get that $g''(x)+\alpha^2e^{2x}g(x)$.
\end{proof}

\begin{proposition}\label{prop:realeigenvalues}
	Let $\Omega$ be a bounded domain in $\RR^2$, 
	and let $f\colon\ol{\Omega}\to\RR$ be a function in 
	$C^2(\Omega)\cap C(\ol{\Omega})$ that satisfies
	$\Delta f+\lambda' f=0$ in $\Omega$, $\frac{1}{|\Omega|}\int_{\Omega}|f(x)|^2dx=1$,
	and $\sup_{x\in \dl \Omega}|f(x)|\le \eps$, where $\eps<1$.
	Then there exists a Dirichlet eigenvalue $\lambda$ of $\Omega$
	such that $|\lambda'-\lambda|\le \lambda\eps$.
\end{proposition}
\begin{proof}
	This is a special case of~\cite[Theorem~1]{MoPa}.
\end{proof}

We are now ready to prove our main result.
\begin{proof}[Proof of Theorem~\ref{thm:mainthm}]
	To simplify notation we set $\lambda^{(N)}\coloneqq\lambda_k(\Pp_N)$ 
	and $\lambda\coloneqq\lambda_k(\DD)$.
	Since the $\lambda$-eigenfunction of~$\DD$ is radially-symmetric, 
	we may assume that, for all sufficiently large~$N$, 
	the $\lambda^{(N)}$-eigenfunction of $\Pp_N$ is dihedrally-symmetric. 
	More precisely, if
	\begin{equation}
	\begin{cases}
	\Delta \phi(z) + \lambda^{(N)}\phi(z) \= 0 \,,\\
	\phi(z) \= 0 \,,\qquad z\in\dl\Pp_N\,,
	\end{cases}
	\end{equation}
	then we may assume that $\phi(e^{2\pi i/N}z)=\phi(\ol{z})=\phi(z)$.
	
	By the general theory developed by Vekua~\cite[(13.5), p.~58]{Ve} any function 
	$\phi$ that satisfies $\Delta\phi+\lambda^{(N)}\phi=0$ in $\Pp_N$ can be 
	represented as
	\begin{equation}  \label{eq:vekua}
	\phi(z) \= a_0J_0(\sqrt{\lambda^{(N)}}|z|)
	+ \mathrm{Re}\int_{0}^{z}
	U(t)J_0\Big(\sqrt{\lambda^{(N)}\ol{z}(z-t)}\Big)dt \,,
	\end{equation}
	where $a_0\in\RR$ and $U\colon\Pp_N\to\CC$ is some holomorphic function.
	Since by assumption $\phi(z)$ is dihedrally-symmetric, we may write 
	$U(f(t))=\widetilde{U}(t^N)/t$, where $\widetilde{U}(0)=0$ and the Taylor 
	series of $\widetilde{U}$ at~0 has real coefficients (we recall that
	$f$ is defined by~\eqref{eq:confmap}). Then
	\begin{align*}
	\phi(f(z)) 
	\= a_0J_0(\sqrt{\lambda^{(N)}}|f(z)|) + \mathrm{Re}\int_{0}^{z}
	f'(t)U(f(t))J_0\Big(\sqrt{\lambda^{(N)}\ol{f(z)}(f(z)-f(t))}\Big)dt \\
	\= a_0J_0(\sqrt{\lambda^{(N)}}|f(z)|) 
	+ c_N\mathrm{Re}\int_{0}^{z} (1-t^N)^{-2/N}\widetilde{U}(t^N)
	J_0\Big(\sqrt{\lambda^{(N)}\ol{f(z)}(f(z)-f(t))}\Big)\frac{dt}{t} \,.
	\end{align*}
	After replacing $z$ and $t$ by $z^{1/N}$ and $t^{1/N}$ respectively and 
	setting $\psi(z)\coloneqq \phi(f(z^{1/N}))$ and
	$V(z) \coloneqq \frac{\widetilde{U}(z)}{(1-z)^{2/N}}$ 
	we obtain
	\begin{equation} \label{eq:mainrepr}
	\psi(z) \= a_0J_0(\rho^{1/2} |z|^{1/N}|F_N(z)|) 
	+ \frac{c_N}{N}\mathrm{Re}\int_{0}^{z}V(t)K(z,t) \frac{dt}{t} \,,
	\end{equation}
	where we set $\rho\coloneqq c_N^2\lambda^{(N)}$ and
	\begin{equation*}
	K(z,t) \coloneqq J_0\Big(\rho^{1/2} |z|^{1/N}
	\sqrt{F_{N}(\ol{z})(F_{N}(z)-(t/z)^{1/N}F_N(t))}\Big)\,.
	\end{equation*}
	Now we make an ansatz that 
	\begin{equation} \label{eq:mainansatz}
	\begin{split}
	\rho &\es\sim \lambda\exp\Big(\frac{\kappa_1}{N}+\frac{\kappa_2}{N^2}+\dots\Big)\,,\\
	V(z) &\es\sim V_0(z)+\frac{V_1(z)}{N}+\frac{V_2(z)}{N^2}+\dots\,.
	\end{split}
	\end{equation}
	where $V_j\colon\DD\to\CC$ are holomorphic and $V_j(0)=0$.
	
	In view of Proposition~\ref{prop:besselasymp} (ii) it is convenient to 
	set $a_0=\frac{c_N}{\lambda^{1/2}J_1(\lambda^{1/2})}$. We will 
	expand~\eqref{eq:mainrepr} as an asymptotic series in powers of $1/N$
	and then recursively compute $\kappa_j$ and $V_j$ using 
	the boundary condition $\psi(z)=0$, $|z|=1$.
	Note that by~\eqref{eq:hypergexp} $F_N(x)=1+O(N^{-2})$ and 
	using~\eqref{eq:hypergexp} and the expansion 
	$(t/z)^{1/N}=\sum_{n\ge 0}\frac{\log^n(t/z)}{n!}N^{-n}$ we get
	\begin{align*}
	K(z,t) &\= \sum_{n=0}^{r}(-\tfrac{\rho}{4})^{n}
	\frac{F_N(\ol{z})^n(F_N(z)-(t/z)^{1/N}F_N(t))^n}{n!^2} + O(N^{-r-1})\\
	&\= \sum_{u,v,w,m}
	\gamma_{u,v,w,m}\Li_{u}(\ol{z})\Li_{v}(z)\Li_{w}(t)\log^m(z/t)N^{-|u|-|v|-|w|-m}+O(N^{-r-1})\,,
	\end{align*}
	where $\gamma_{u,v,w,m}$ are coefficients that depend on $\kappa_i$
	and the summation is over words $u,v,w\in x_0X^{\times}x_1$ and $m\ge0$ 
	satisfying $|u|+|v|+|w|+m\le r$.
	We get a similar expression (involving only products 
	$\Li_{u}(z)\Li_{v}(\ol{z})$) after expanding $a_0J_0(\rho^{1/2}|F_N(z)|)$
	using~\eqref{eq:besselexp}.
	
	We claim that $\kappa_i$ and~$V_i(z)$ can be calculated inductively by 
	comparing the coefficients of the $1/N$-expansion. Indeed, comparing the 
	coefficients of $1/N$ we see that
	\[\frac{\kappa_1}{2} - \re \int_{0}^{z}V_0(t)\frac{dt}{t} \= 0\,, \qquad |z|=1\,,\]
	so that $\kappa_1=V_0(z)=0$. 
	In general, assume that $\kappa_i\in \MZV_i[\lambda]$
	and $V_{i-1}(z)=\Li_{v_{i-1}}(z)$ for $i=1,\dots,k$, 
	where $v_i\in\MZV[\lambda]\ang{X}x_1$ is of total weight $i$
	(we define the total weight of $\lambda^azw$, where $z\in\MZV_k$ 
	to be $k+|w|$). 
	Note that by Lemma~\ref{lem:convolution}
	\begin{align*}
	\int_{0}^{z}\Li_{u}(\ol{z})\Li_{v}(z)\Li_{w}(t)\log^m(z/t)\frac{dt}{t}
	\= m!\Li_{u}(\ol{z})\Li_{w'}(z)\,,
	\end{align*}
	where $w'=v\shuffle x_0^{m+1}w$. Thus, when comparing the coefficients
	of $N^{-k-1}$, we need to ensure an identity of the form
	\[\frac{\kappa_{k+1}}{2}-\re\int_{0}^{z}V_k(t)\frac{dt}{t} - 
	\sum_{u,v}\gamma_{u,v}\re\Li_{u}(z)\Li_{v}(\ol{z}) \= 0\,,\qquad |z|=1\,,\]
	where the terms in the sum only depend on already computed quantities 
	$\kappa_1,\dots,\kappa_{k}$ and $V_0(z),\dots,V_{k-1}(z)$
	and all have total weight $k+1$ 
	(where we define the total weight of $\gamma\Li_{u}(z)\Li_{v}(\ol{z})$ for $\gamma\in \MZV_k[\lambda]$ to be $|u|+|v|+k$). 
	By Proposition~\ref{prop:polylogproduct} this amounts to setting
	\begin{align*}
	x_0v_{k}     &\= -\sum_{u,v}\gamma_{u,v}(\alpha(u,v)+\beta(u,v))\,,\\
	\frac{\kappa_{k+1}}{2} &\= \sum_{u,v}\gamma_{u,v}A_{u,v}\,.
	\end{align*}
	This indeed can be done since by assumption the elements $u$ and $v$ are 
	convergent and hence $\alpha(u,v),\beta(u,v)\in x_0\CC\ang{X}x_1$.
	This gives an explicit algebraic recursion for $\kappa_n$ and $V_n(z)$
	that shows, in particular, that $\kappa_n \in \MZV_{n}[\lambda]$.
	In Table~\ref{tab:vn} we list the functions $V_n(z)$ 
	for $n\le 4$. We also note that the coefficients $\kappa_n$ 
	vanish for $n\le 4$.
	\begin{table}
		{\def\arraystretch{1.5}
			\begin{tabularx}{\linewidth}{r|R}
				$n$ & $V_n(z)$ \\
				\hline
				$0$   & $0$ \\
				$1$   & $2\Li_1(z)$ \\
				$2$   & $(\tfrac{\lambda}{2}-2)\Li_2(z)+4\Li_{1,1}(z)$ \\
				$3$   & $(\tfrac{\lambda^2}{16}-\lambda+2)\Li_3(z)+(3\lambda-12)\Li_{1,2}(z)+(\lambda-4)\Li_{2,1}(z)+2^3\Li_{1,1,1}(z)$ \\
				$4$   & $
				(\tfrac{\lambda^3}{192}-\tfrac{\lambda^2}{8}-\tfrac{\lambda}{2}-2)
				\Li_4(z)
				+(\tfrac{\lambda^2}{8}-2\lambda+4)\Li_{3,1}(z)
				+(\tfrac{\lambda^2}{4}-4\lambda+12)\Li_{2,2}(z)$
				$+(\tfrac{5\lambda^2}{8}-8\lambda+28)\Li_{1,3}(z)
				+(2\lambda-8)\Li_{2,1,1}(z)+(6\lambda-24)\Li_{1,2,1}(z)$
				$+(14\lambda-56)\Li_{1,1,2}(z)+2^4\Li_{1,1,1,1}(z)+2Z_3\Li_{1}(z)$ \\
				\hline
		\end{tabularx}}
		\smallskip
		\caption{The functions $V_n(z)$ for $n\le 4$}
		\label{tab:vn}
	\end{table}
	
	We still need to verify that the ansatz~\eqref{eq:mainansatz}
	indeed gives an asymptotic expansion for $\lambda^{(N)}$.
	To see this, note that plugging a truncated solution
	for the boundary condition $\psi(z)=0$, $|z|=1$ back into~\eqref{eq:vekua} 
	we obtain a sequence of functions $\phi^{N,r}\colon\Pp_N\to\RR$, 
	and numbers $\lambda^{N,r}$. 
	The numbers $\lambda^{N,r}$ converge to $\lambda_k$ as $N\to \infty$ 
	for each fixed $r$ and the functions $\phi^{N,r}$ satisfy 
	\[\Delta\phi^{N,r}(z)+\lambda^{N,r}\phi^{N,r}(z)=0
	\,,\qquad  z\in\Pp_N\,,\]
	together with $\|\phi^{N,r}\|_2\gg 1$ 
	and $\|\phi^{N,r}|_{\partial \Pp_N}\|_{\infty}\ll_{r} N^{-r-1}$. 
	Therefore, applying Proposition~\ref{prop:realeigenvalues} shows that 
	$|\lambda^{N,r}-\lambda_k(\Pp_N)|\ll_{r} N^{-r-1}$, so that $\lambda^{N,r}$ 
	indeed give an asymptotic expansion for $\lambda_k(\Pp_N)$.
		
	Finally, the coefficients $C_n(\lambda)$ are related to $\kappa_n(\lambda)$ 
	by the generating series identity
	\begin{equation*} \label{eq:kappadef}
		\exp\big(\kappa_1(\lambda)z+\kappa_2(\lambda)z^2+\dots\big) 
		\=
		\frac{\Gamma(1-z)^2\Gamma(1+2z)}{\Gamma(1+z)^2\Gamma(1-2z)}
		\Big(1+\sum_{n\ge1}{C_n(\lambda)z^n}\Big)\,,
	\end{equation*}
	and using
	\begin{equation} \label{eq:gammaprodexp}
		\frac{\Gamma^2(1+z)\Gamma(1-2z)}{\Gamma^2(1-z)\Gamma(1+2z)}
		\=\exp\Big(\sum_{k\ge 1}\zeta(2k+1)\frac{4(4^k-1)z^{2k+1}}{2k+1}\Big)
	\end{equation}
	we obtain that $C_n\in\MZV_n[\lambda]$.
\end{proof}

\smallskip
\section{Explicit formulas for $C_n(0)$ and $C_n'(0)$}
\label{sec:explicitformulae}	
\begin{proof}[Proof of Theorem~\ref{thm:mainthm2}]
	We follow the algebraic recursion for $\kappa_{k}$ and $V_k(z)$ given in 
	the proof of Theorem~\ref{thm:mainthm} ignoring all the terms 
	involving $\lambda^k$ for $k\ge 2$ (in other words we work modulo the ideal 
	generated by~$\lambda^2$). 
	For this we write 
	\[V(z)=V^{(0)}(z)+\lambda V^{(1)}(z)+O(\lambda^2)\]
	and
	\[\kappa \coloneqq \sum_{n\ge1}\frac{\kappa_n}{N^n}=
	\kappa^{(0)}+\lambda\kappa^{(1)}+O(\lambda^2)\,.\]
	Note that~\eqref{eq:besselexp} implies that
	\[
	\frac{J_0(\lambda^{1/2} e^{x})}{\lambda^{1/2}J_1(\lambda^{1/2})} \= 
	-x + \frac{\lambda}{4}(x+1+(x-1)e^{2x})+O(\lambda^2)\,,
	\]
	so that
	\begin{equation*}
	\begin{split}
	a_0
	J_0(\lambda^{1/2}e^{\kappa/2+\log|F_N(z)|}) \= 
	c_N\Big(-\kappa/2-\log|F_N(z)| + \frac{\lambda}{4}(\kappa/2+\log|F_N(z)|+1\\
	+(\kappa/2+\log|F_N(z)|-1)e^{\kappa}|F_N(z)|^2)\Big)+O(\lambda^2)\,.
	\end{split}
	\end{equation*}
	Similarly, we calculate the kernel $K(z,t)$ to order $O(\lambda^2)$ as
	\begin{align*}
	K(z,t) &\= 1-\frac{\lambda}{4}e^{\kappa} F_N(\ol{z})(F_N(z)-(t/z)^{1/N}F_N(t)) + O(\lambda^2)\,.
	\end{align*}
	If we first look at the boundary condition modulo $O(\lambda)$,
	it reads
	\[\tfrac12\kappa^{(0)}+\log|F_N(z)|\= 
	N^{-1}\re\int_{0}^{z}V^{(0)}(t)\frac{dt}{t}\,,\qquad |z|=1\,.\]
	This clearly implies $N^{-1}\int_{0}^{z}V^{(0)}(t)\frac{dt}{t}\=\log F_N(z)$
	and $\kappa^{(0)}=0$.
	Using this we can rewrite the boundary condition for the linear 
	term in $\lambda$ as 
	\begin{align*}
	\tfrac{1}{2}\kappa^{(1)} 
	- \frac{1}{4}(\log|F_N(z)|+1+(\log|F_N(z)|-1)|F_N(z)|^2)
	- N^{-1}\re \int_{0}^{z}V^{(1)}(t)\frac{dt}{t} \\
	\=  -\frac{1}{4} N^{-1}\re \int_{0}^{z}V^{(0)}(t)
	F_N(\ol{z})(F_N(z)-(t/z)^{1/N}F_N(t))\frac{dt}{t}\,,\qquad |z|=1\,.
	\end{align*}
	Since $N^{-1}\int_{0}^{z}V^{(0)}(t)\frac{dt}{t}=\log F_N(z)$,
	we have $V^{(0)}(t)=Nt\frac{F_N'(t)}{F_N(t)}$ and thus
	\begin{equation*}
	\tfrac{1}{2}\kappa^{(1)} - \frac{1}{4}\log|F_N(z)|
	- N^{-1}\re \int_{0}^{z}V^{(1)}(t)\frac{dt}{t} 
	\=  \frac{1}{4}\Big(1-|F_N(z)|^2+\re F_N(\ol{z})\widetilde{F}_N(z)\Big)
	\end{equation*}
	where we denote $\widetilde{F}_N(z)=\int_{0}^{z}(t/z)^{1/N}F_N'(t)dt$.
	Integrating over $|z|=1$ leads to
	\begin{equation} \label{eq:kappa1int}
	2\kappa^{(1)} \= \int_{0}^{1}\Big(1-|F_N(e^{2\pi i x})|^2
	+F_N(e^{-2\pi ix})\widetilde{F}_N(e^{2\pi i x})\Big)dx\,.
	\end{equation}
	Our goal is to rewrite the right hand side of~\eqref{eq:kappa1int}
	as a hypergeometric series.
	For this we will make use of the integral representation for $F_N(z)$
	\[F_N(z)
	\= 1+\frac{1}{N}\int_{0}^{1}t^{1/N}((1-tz)^{-2/N}-1)\frac{dt}{t}\,.\]
	First, we plug this representation into the definition of $\widetilde{F}_N(z)$ to obtain
	\begin{align*}
	\widetilde{F}_N(z) &\= \frac{2}{N^2}\int_{0}^{z}(x/z)^{1/N}
	\Big(\int_{0}^{1}t^{1/N}(1-xt)^{-2/N-1}dt\Big)dx\\
	&\= \frac{2}{N^2}\int_{0}^{1}\int_{0}^{1}
	(t_1t_2)^{1/N}z(1-zt_1t_2)^{-2/N-1}dt_1dt_2\\
	&\=-\frac{2}{N^2}\int_{0}^{1}t^{1/N}z(1-zt)^{-2/N-1}\log t dt\,.
	\end{align*}
	From this, by changing the order of integration, we calculate
	\begin{align*}
	\int_{0}^{1}1-|F_N(e^{2\pi i x})|^2dx 
	&\= -\frac{1}{N^2}\int_{0}^{1}\int_{0}^{1}(t_1t_2)^{1/N-1}
	(G_N(t_1t_2)-1)dt_1dt_2\,,\\
	\int_{0}^{1}F_N(e^{2\pi i x})\widetilde{F}_N(e^{-2\pi i x})dx 
	&\= -\frac{1}{N^2}\int_{0}^{1}\int_{0}^{1}(t_1t_2)^{1/N-1}
	t_1t_2G_N'(t_1t_2)\log t_1dt_1dt_2\,,
	\end{align*}
	where
	\[G_N(x)
	\={}_{2}F_{1}(2/N,2/N,1;x)\=\sum_{n\ge0}\frac{(2/N)_n^2}{n!^2}x^n\,.\]
	Next, using the easily verified identity
	\[\int_{0}^{1}\int_{0}^{1}f(xy)\log^kx\frac{dx}{x}\frac{dy}{y} 
	\= -\frac{1}{k+1}\int_{0}^{1}f(t)\log^{k+1} t \frac{dt}{t}\]
	we can rewrite the above double integrals as single integrals.
	Plugging the resulting expressions back into~\eqref{eq:kappa1int} we get
	\begin{equation*}
	2\kappa^{(1)} \= \frac{1}{N^2}\int_{0}^{1}
	((G_N(t)-1)\log t+\tfrac{t}{2} G_N'(t)\log^2t)t^{1/N}\frac{dt}{t}\,.
	\end{equation*}
	Finally, expanding $G_N(t)$ as a power series in $t$ and integrating the 
	above identity term-by-term we obtain
	\begin{equation} \label{eq:kappa1}
	\kappa^{(1)} 
	\= -\frac{1}{2N^{3}}\sum_{n\ge1}\frac{(2/N)_n^2}{n!^2(1/N+n)^3}\,,
	\end{equation}
	which immediately implies~\eqref{eq:constterm}.
\end{proof}
	
Finally, let us prove Theorem~\ref{thm:oddzeta}. For this we will need the 
following lemma to evaluate the right hand side of~\eqref{eq:kappa1}
in terms of gamma function and its derivatives.
\begin{lemma}
	For all $z\not \in \frac{1}{2}\ZZ$ we have
	\begin{equation}\label{eq:trigammafinite}
		\begin{split}
			&\sum_{n=0}^{m}\frac{(2z)_n^2z^3}{n!^2(z+n)^3}\frac{(1+m)_n(-m)_n}{(1+2z+m)_n(2z-m)_n} \\
			&\= 
			\frac{(1+2z)_m(1-z)_m^2}{(1-2z)_m(1+z)_m^2}\Big(1+\sum_{j=1}^{m}\frac{j(j-2z)}{(j-z)^2}-\sum_{j=1}^{m}\frac{j(j+2z)}{(j+z)^2} \Big)
		\end{split}
	\end{equation}
\end{lemma}
\begin{proof}
	We will prove this identity by induction on~$m$, the case $m=0$ being 
	trivial. First, we divide both sides by the product of the Pochhammer 
	symbols on the right to get an equivalent identity
	\begin{align*}
	&\sum_{n=0}^{m}T_{n,m}\=\frac{1}{z^3}
	\Big(1+\sum_{j=1}^{m}\frac{j(j-2z)}{(j-z)^2}-\sum_{j=1}^{m}\frac{j(j+2z)}{(j+z)^2}\Big)\,,
	\end{align*}
	where we denote
	\[T_{n,m} \coloneqq \frac{(1-2z)_m(1+z)_m^2}{(1+2z)_m(1-z)_m^2}
	\frac{(2z)_n^2}{n!^2(z+n)^3}\frac{(1+m)_n(-m)_n}{(1+2z+m)_n(2z-m)_n}\,.\]
	
	Then it is enough to show that
	\[\sum_{n=0}^{m+1}(T_{n,m+1}-T_{n,m}) \= 
	-\frac{4(m+1)}{(m+1-z)^2(m+1+z)^2}\,.\]
	First, an elementary calculation shows that for $0\le n \le m+1$ we have
	\begin{align*}
	T_{n,m+1}-T_{n,m} \= \frac{4z(2z)_n^2}{n!^2(z+n)}
	\frac{(1-2z)_{m+1}(1+z)_m^2}{(1-z)_{m+1}^2(1+2z)_m}
	\frac{(1+m)_{n}(-m-1)_n}{(1+2z+m)_{n+1}(2z-m-1)_{n+1}}\,.
	\end{align*}
	Therefore, it suffices to show that for $m\ge 1$ we have
	\[\sum_{n=0}^{m}D_{n,m} \= 1\,,\]
	where 
	\[D_{n,m} \= -\frac{z(m-z)^2(2z)_n^2}{n!^2(z+n)m}
	\frac{(1-2z)_{m}(1+z)_{m}^2}{(1+2z)_{m+n}(1-z)_{m}^2}
	\frac{(m)_{n}(-m)_n}{(2z-m)_{n+1}}\,.\]
	The last identity can be easily proved using the Wilf-Zeilberger 
	method~\cite{WZ}. Explicitly, using the identities
	\begin{align*}
		\frac{D_{n,m+1}}{D_{n,m}} &\= \frac{(m+1+z)^2(m-n-2z)(m+n)(m+1)}{(m-z)^2(m+n+1+2z)m(m-n+1)} \\
		\frac{D_{n+1,m}}{D_{n,m}} &\= \frac{(n+z)(n+2z)^2(m-n)(m+n)}{(n+1+z)(m-n-1-2z)(m+n+1+2z)(n+1)^2}
	\end{align*}
	one can verify that
	\begin{equation} \label{eq:telescopy}
	D_{n,m+1} - D_{n,m} = G_{n+1,m}-G_{n,m}\,,\qquad n=0,\dots,m+1\,,
	\end{equation}
	where $G_{n,m} = D_{n,m}R(n,m)$ and 
	\[R(n,m) \= \frac{n^2(2m + 1)(z + n)(m - n - 2z)(2m^2z + (m+z)^2 - (n+z)^2 + m + n + 2z)}{4m^2(m + 1)^2(n-m-1)(m - z)^2z}\,.\]
	(We regularize $G_{n,m}$ for $n=m+1$ by canceling the factor $(n-m-1)$ in 
	the denominator of $R(n,m)$ with the factor $(n-1-m)$ coming from the 
	Pochhammer symbol $(-m)_n$ in the definition of $D_{n,m}$.)
	Finally, the identity $\sum_{n=0}^{m}D_{n,m} = 1$ then follows by induction 
	on $m$ from~\eqref{eq:telescopy} together with $G_{0,m}=G_{m+2,m}=0$.
\end{proof}
\begin{corollary}
For all $z\in\CC$ with $|z|<1$ we have
\begin{equation} \label{eq:trigamma}
	\sum_{n\ge0}\frac{(2z)_n^2z^3}{n!^2(z+n)^3}\=
	\frac{\Gamma^2(1+z)\Gamma(1-2z)}{\Gamma^2(1-z)\Gamma(1+2z)} \Big(1+z^2\psi^{(1)}(1+z)-z^2\psi^{(1)}(1-z)\Big)\,,
\end{equation}
where $\psi^{(k)}(z) = \frac{d^{k+1}}{dz^{k+1}}\log\Gamma(z)$ is the $k$-th 
polygamma function. 
\end{corollary}
\begin{proof}
For $\re z<1$ and $z\not\in\frac{1}{2}\ZZ$ simply take the limit $m\to+\infty$ 
in~\eqref{eq:trigammafinite} using the fact that
$(x)_n \= \frac{\Gamma(x+n)}{\Gamma(x)}\sim \frac{\Gamma(n)n^{x}}{\Gamma(x)}$, $n\to\infty$, together with 
	\[\sum_{j=1}^{m}\frac{j(j\pm 2z)}{(j\pm z)^2} \= 
	m - z^2\psi^{(1)}(1\pm z)+z^2\psi^{(1)}(1+m\pm z)\,.\] 
Noting that both sides of~\eqref{eq:trigamma} are non-singular also for 
$z=0,\pm 1/2$ we obtain the claim for all $|z|<1$.
\end{proof}

\begin{proof}[Proof of Theorem~\ref{thm:oddzeta}]
	The claim for $C_n(0)$ immediately follows from~\eqref{eq:constterm}
	and~\eqref{eq:gammaprodexp}. Differentiating the classical identity
	\[	\log\Gamma(1-z) \= \gamma z +\sum_{n\ge2}\zeta(n) \frac{z^n}{n}\,,\]
	twice shows that
	\[z^2\psi^{(1)}(1-z)-z^2\psi^{(1)}(1+z) \= 
	\sum_{k\ge 1}4k\zeta(2k+1)z^{2k+1}\,,\] 
	and thus~\eqref{eq:constterm},~\eqref{eq:gammaprodexp} and~\eqref{eq:trigamma} together imply the claim for $C_n'(0)$.
\end{proof}


\begin{thebibliography}{99}
\bibitem{AF} P.~Antunes, P.~Freitas, \emph{New bounds for the principal Dirichlet eigenvalue of planar regions}, Experiment. Math. \textbf{15}, pp.~333--342 (2006).

%\bibitem{Ba} W.N.~Bailey, \emph{Generalised Hypergeometric Series}. Cambridge 
%University Press, 1935. 

\bibitem{BJMR} D.~Berghaus, R.~Jones, H.~Monien, D.~Radchenko, 
\emph{Computation of the lowest Laplacian eigenvalue of regular polygons}, in preparation.

\bibitem{BBV} J.~Bl\"{u}mlein, D.~J.~Broadhurst, and J.~A.~M.~Vermaseren, \emph{The multiple {z}eta {v}alue {d}ata {m}ine}, Comput. Phys. Comm.~\textbf{181} (3): 582--625 (2010).

\bibitem{Bo} M.~Boady, \emph{Applications of symbolic computation to the calculus of moving surfaces}, PhD thesis, Drexel University, Philadelphia, PA, 2015.

\bibitem{Br1} F.~C.~S.~Brown, \emph{Polylogarithmes multiples uniformes en une variable}, C.~R.~Acad.~Sci. Paris, Ser. I, \textbf{338}, pp.~527--532 (2004).

\bibitem{Br2} F.~C.~S.~Brown, \emph{Single-valued periods and multiple zeta values}, Forum of Mathematics, Sigma, Vol.~\textbf{2}, no.~\textbf{e25} p.~37 (2014).

\bibitem{BF} J.~I.~Burgos~Gil, J.~Fres\'an, \emph{Multiple zeta values: from numbers to motives}, Clay Math. Proceedings, to appear.
	
\bibitem{Cha} I.~Chavel, \emph{Eigenvalues in Riemannian geometry}, with a chapter by Burton Randol and an appendix by Jozef Dodziuk, Academic Press, Inc.    (Harcourt Brace Jovanovich, Publishers), Orlando, San Diego, New York, London, Toronto, Montreal, Sydney, Tokyo, 1984.

\bibitem{CGR} S.~Charlton, H.~Gangl, D.~Radchenko, \emph{On functional equations for Nielsen polylogarithms}, \arXiv{1908.04770}.

\bibitem{FoHeMo} L.~Fox, P.~Henrici, C.~B.~Moler, \emph{Approximations and bounds for eigenvalues of elliptic operators}, SIAM J.~Numer.~Anal., \textbf{4}, pp.~89--102 (1967).

\bibitem{GS1} P.~Grinfeld, G.~Strang, \emph{The Laplacian eigenvalues of a polygon}, Computers and Mathematics with Applications, \textbf{48}, pp.~1121--1133 (2004).

\bibitem{GS2} P.~Grinfeld, G.~Strang, \emph{Laplace eigenvalues on regular polygons: A series in 1/N.} J. Math. Anal. Appl., \textbf{385}, pp.~135--149 (2012).

\bibitem{NPH} M.~Hoang~Ngoc, M.~Petitot, J.~van~der~Hoeven, \emph{Polylogarithms and shuffle algebra}, in FPSAC '98, Toronto, Canada, 1998.

\bibitem{Jo} R.~Jones, \emph{The fundamental Laplacian eigenvalue of the regular polygon with Dirichlet boundary conditions}, \arXiv{1712.06082}.

\bibitem{Julia} J.~Bezanson, A.~Edelman, S.~Karpinski, V.B.~Shah, \emph{Julia: A fresh approach to numerical computing}, SIAM {R}eview \textbf{59 (1)}, pp.~65--98 (2017).

\bibitem{Ko} K.~S.~K\"{o}lbig, \emph{Nielsen's generalized polylogarithms},
SIAM J. Math. Anal., \textbf{17(5)}, pp.~1232--1258, (1986).

\bibitem{Mo}  L.~Molinari, \emph{On the ground state of regular polygonal billiards}, J. Phys. A: Math. Gen., \textbf{30(18)}, pp.~6517--6524 (1997).

\bibitem{MoPa} C.~B.~Moler, L.~E.~Payne, \emph{Bounds for eigenvalues and eigenvectors of symmetric operators}, SIAM J.~Numer.~Anal.~\textbf{5}, pp.~64--70 (1968).

\bibitem{Ni} C.~Nitsch, \emph{On the first Dirichlet Laplacian eigenvalue of regular polygons}. Kodai Mathematical Journal, \textbf{37}, pp.~595--607 (2014).% arXiv:1403.6709.

\bibitem{Oi} V.~K.~Oikonomou. \emph{Casimir energy for a regular polygon with Dirichlet boundaries}, \arXiv{1012.5376}.

\bibitem{So} A.~Yu.~Solynin, \emph{Isoperimetric inequalities for polygons and dissymetrization}, Algebra i Analiz, \textbf{4:2}, pp.~210--234 (1992); St. Petersburg Math. J., \textbf{4:2}, pp.~377--396 (1993).     

\bibitem{SAGE}
W.\thinspace{}A. Stein et~al., \emph{{S}age {M}athematics {S}oftware
({V}ersion 9.1)}, The Sage Development Team, 2020, {\tt
http://www.sagemath.org}.

\bibitem{Ve} I.~N.~Vekua, \emph{New method for solving elliptic equations}, North-Holland Publishing Company, Amsterdam, 1967.

\bibitem{Vi} M.~A.~Virasoro, \emph{Alternative constructions of crossing-symmetric amplitudes with Regge behaviour}, Phys.~Rev., \textbf{177}, pp.~2309--2311 (1969).

\bibitem{W} M.~Waldschmidt, \emph{Lectures on Multiple Zeta Values}, IMSC 2011.

\bibitem{WZ} H.~Wilf, D.~Zeilberger, \emph{Rational functions certify combinatorial identities}, J.~Amer.~Math.~Soc.~\textbf{3}, pp.~147--158 (1990).

\bibitem{Za} D.~Zagier, \emph{Values of zeta functions and their applications}, in: Proceedings of ECM~1992, Progress in Math.~\textbf{120}, pp.~497--512 (1994).

\bibitem{Ze}  S.~Zelditch, \emph{Eigenfunctions of the Laplacian on a Riemannian Manifold}, CBMS Regional Conference Series in Mathematics Vol. 125, 2017.
\end{thebibliography}
\end{document}